\documentclass[11pt]{amsart}
\usepackage[utf8]{inputenc}
\usepackage{amsmath,amssymb,amsthm}
\usepackage{enumerate}
\usepackage{xcolor, wasysym}
\usepackage{multirow}
\usepackage{comment}
\usepackage[left=1.2in,right=1.2in]{geometry}
\usepackage{graphicx,import}

\usepackage[colorlinks,citecolor=black,linkcolor=black]{hyperref}
\usepackage[nameinlink]{cleveref}

\newcounter{hcomments}

\newcounter{gcomments}

\newcounter{scomments}

\theoremstyle{plain}
\newtheorem{THM}{Theorem}[section]
\newtheorem{thm}[THM]{Theorem}
\newtheorem{prop}[THM]{Proposition}
\newtheorem{lemma}[THM]{Lemma}
\newtheorem{cor}[THM]{Corollary}

\theoremstyle{definition}
\newtheorem{DEF}[THM]{Definition}

\newtheorem{innercustomthm}{Theorem}
\newenvironment{customthm}[1]
  {\renewcommand\theinnercustomthm{#1}\innercustomthm}
  {\endinnercustomthm}

\DeclareMathOperator{\id}{id}
\DeclareMathOperator{\Aut}{Aut}
\DeclareMathOperator{\Out}{Out}
\DeclareMathOperator{\Inn}{Inn}

\DeclareMathOperator{\diam}{diam}

\newcommand{\N}{\mathbb{N}}

\newcommand{\cF}{\mathcal{F}}

\newcommand{\cV}{\mathcal{V}}
\newcommand{\cU}{\mathcal{U}}

\def\G{{\Gamma}}

\newcommand{\arr}{\rightarrow}

\newcommand{\<}{\langle}
\renewcommand{\>}{\rangle}

\title[$\Aut(F_\infty)$ and $\Out(F_\infty)$ are coarsely bounded]{The Automorphism Group of the Infinite-Rank Free Group is Coarsely Bounded}
\author{George Domat, Hannah Hoganson, Sanghoon Kwak}
\date{}

\begin{document}

\maketitle 

\begin{abstract}
    We prove that the full automorphism group and the outer automorphism group of the free group of countably infinite rank are coarsely bounded. That is, these groups admit no continuous actions on a metric space with unbounded orbits and have the quasi-isometry type of a point.
\end{abstract}

\section{Introduction}

Classical geometric group theory concerns itself with finitely generated and then compactly generated groups.  In these settings, groups have a well-defined quasi-isometry type that allows one to study them from a geometric perspective. Recent work of Roe and Rosendal \cite{roe2003lectures,rosendal2022} has expanded the tools of coarse geometry to the broader class of non-compactly generated Polish groups. They introduce a generalization of compactness called \emph{coarse boundedness}. Groups that have a coarsely bounded generating set admit a well-defined quasi-isometry type. Globally coarsely bounded groups are quasi-isometric to a point, so fail to admit any interesting geometry. 

These advances have coincided with a recent burst of interest in \emph{big mapping class groups}, that is, mapping class groups of infinite-type surfaces. Unlike their finite-type analogues, these groups are uncountable, non-compactly generated, Polish groups. Mann-Rafi \cite{mann2022large} began the study of the coarse geometry of these groups using the framework of coarse boundedness. 

In the finite-type setting, the study of mapping class groups and the outer automorphism groups of free groups, $\Out(F_{n})$, are closely intertwined. Indeed, the fundamental group of a punctured surface is a free group and the mapping class group acts faithfully on it. This gives an injective map from the mapping class group to $\Out(F_{n})$ for an appropriate $n$. This begs two natural questions: What is the ``big" analogue of $\Out(F_{n})$ and how does it connect to the mapping class group?

The fundamental group of any infinite-type surface with infinite genus is the free group of countably infinite rank, $F_{\infty}$. So, a natural first guess is that $\Out(F_{\infty})$ serves as the ``big" analogue of $\Out(F_{n})$. Indeed, big mapping class groups do still have injective maps into $\Out(F_{\infty})$. In this paper, we prove the following.

\begin{customthm}{A}\label{THM:Main}
    The groups $\Aut(F_{\infty})$ and $\Out(F_{\infty})$ are both coarsely bounded.
\end{customthm}

This theorem states that each of these two groups fail to admit any continuous actions on a metric space with unbounded orbits, and in fact are both quasi-isometric to a point. This is in sharp contrast with the groups $\Aut(F_n)$ and $\Out(F_n)$, which are \emph{not} coarsely bounded for any $n\geq 2$. Given the theorem above one may be tempted to ask about $\Aut(F_{\kappa})$ for $\kappa$ an uncountable cardinal. We will not consider these groups because the natural topologies that we define in \Cref{ssec:topology} are not first countable and thus not metrizable. 

Mann--Rafi \cite{mann2022large} prove that a large class of infinite-type mapping class groups are in fact \emph{not} coarsely bounded. This suggests that the whole group $\Out(F_{\infty})$ is ``too big" to serve as a satisfying analogue of $\Out(F_{n})$ for these surfaces. Recent work of Algom-Kfir--Bestvina \cite{AB2021} proposes a different analogue via groups of proper homotopy equivalences of infinite-type graphs. In \cite{DHK2022} we prove that some of these new groups are also \emph{not} coarsely bounded. \Cref{THM:Main}, together with those results, gives further evidence that groups of proper homotopy equivalences are the ``correct" analogues of big mapping class groups.

\section{Preliminaries}

\subsection{The Countably Infinite Rank Free Group}

Let $F_{\infty}$ denote the free group of countably infinite rank. A presentation for $F_{\infty}$ is given by the generators $\{a_i\}_{i\in \N}$ with no relations.  We call the $\{a_i\}_{i\in \N}$ the \emph{standard basis} of $F_{\infty}$. 

\begin{DEF}
  A subgroup $G$ of $F_{\infty}$ is a \textbf{free factor} if there is another subgroup $P$ such that $F_{\infty}=G \ast P$.
\end{DEF}

We use $A_{n}$ to denote the subgroup generated by $\{a_i\}_{i=1}^n$. Each subgroup $A_n$ is a free factor of $F_{\infty}$, and we use $B_{n}$ to denote the complementary free factor. That is, $B_n=\<a_i\>_{i=n+1}^{\infty}$.  

We study the group of automorphisms, $\Aut(F_{\infty})$, of $F_{\infty}$ and the group of outer automorphisms, $\Out(F_{\infty})$. These groups fit into the following short exact sequence, where $\Inn(F_{\infty})$ is the group of automorphisms given by conjugation actions.

\begin{align*}
    1 \longrightarrow \Inn(F_{\infty}) \longrightarrow \Aut(F_{\infty}) \longrightarrow \Out(F_{\infty}) \longrightarrow 1.
\end{align*}

\subsection{Topologies on $\boldsymbol{\Aut}\mathbf{(F_{\boldsymbol{\infty}})}$}
\label{ssec:topology}

The action of $\Aut(F_{\infty})$ on $F_{\infty}$ allows us to define two natural topologies on $\Aut(F_{\infty})$. Equipping $F_{\infty}$ with the discrete topology, we can define the compact-open topology on $\Aut(F_{\infty})$ via the sub-base given by the sets of the form 
\begin{align*}
    \cV_{K,U} = \left\{f \in \Aut(F_{\infty})\big\vert f(K) \subseteq U\right\},
\end{align*}
where $K$ is any finite subset of $F_\infty$ and $U$ is any subset of $F_{\infty}$.

We can also consider the \emph{permutation topology} on $\Aut(F_{\infty})$ arising from this action. Basis elements for this topology are given by the sets of the form
\begin{align*}
    \cU_{K,g} = \left\{f \in \Aut(F_{\infty})\big\vert f\vert_{K} = g\vert_{K}\right\},
\end{align*}
where $K$ is a finite subset of $F_{\infty}$ and $g\in\Aut(F_{\infty})$. This topology is second-countable and supports a complete metric so that $\Aut(F_{\infty})$ has the structure of a Polish group.

\begin{prop}\label{prop:topologies}
The compact-open topology on $\Aut(F_{\infty})$ is equivalent to the permutation topology on $\Aut(F_{\infty})$.
\end{prop}

\begin{proof}
    Let $\cV = \cap_{i=1}^{n} \cV_{K_{i},U_{i}}$ be a basis element for the compact-open topology and let $g \in \cV$. Then the set of the form $\cU = \cU_{\cup_{i}K_{i},g}$ is a basis element for the permutation topology such that $g \in \cU \subset \cV$. Thus we see that the compact-open topology is coarser than the permutation topology.
    
    Now let $\cU = \cU_{K,g}$ be a basis element for the permutation topology and let $g' \in \cU_{K,g}$. The set of the form $\cV = \bigcap_{a_{i}\in K} \cV_{\{a_{i}\},\{g(a_{i})\}}$ is a basis element for the compact-open topology such that $g' \in \cV \subset \cU$ (in fact, $\cV=\cU$). We conclude that the permutation topology is also coarser than the compact-open topology and hence equivalent.
\end{proof}

Throughout the proofs below we will think of $\Aut(F_n)$ equipped with the permutation topology. Because $\Aut(F_{\infty})$ is a topological group we will only focus on neighborhood bases about the identity. Any such basis element $\cU_{K,\id}$ contains a basis element of the form $\cU_{n} := \cU_{\{a_{1},\ldots,a_{n}\},\id}$. Indeed, as the set $K$ is finite, every word within $K$ can be written using only finitely many of the standard basis elements of $F_{\infty}$. Thus $\cU_{n} \subset \cU_{K,\id}$ for sufficiently large $n$.

We further endow $\Out(F_{\infty})$ with the quotient topology. Since the kernel, $\Inn(F_{\infty})$, is a closed subgroup of $\Aut(F_{\infty})$ we have that $\Out(F_{\infty})$ is again a Polish group.

\subsection{Coarse Boundedness}

Instead of the formal definition of coarse boundedness for general coarse spaces, we only give the relevant equivalent definitions for Polish groups. We refer the reader to \cite[Chapter 2]{rosendal2022} for more details.

\begin{DEF}[{\cite[Proposition 2.15]{rosendal2022}}]
  \label{PROP:RosendalCB}
    Let $A$ be a subset of a Polish group $G$. Then we say that $A$ is \textbf{coarsely bounded (CB)} in $G$ if one of the following equivalent conditions is satisfied.
    \begin{enumerate}[(1)]
        \item (Rosendal's Criterion)
            For every neighborhood $\mathcal{U}$ of the identity in $G$, there is a finite subset $\mathcal{F}$ of $G$ and some $n \geq 1$ such that $A \subset (\mathcal{FU})^{n}$.
        \item
            For every continuous action of $G$ on a metric space $X$ and every $x \in X$, $\diam(A \cdot x) < \infty$.
    \end{enumerate}
  We say $G$ is \textbf{coarsely bounded (CB)} if $G$ is coarsely bounded in $G$ itself.
  \end{DEF}

\section{Proof of Main Theorem}

We first need a short lemma on free factors of free groups and reproduce the proof given in \cite[Lemma 7.2]{DHK2022}.

\begin{lemma} \label{LEM:subfreefact}
    Let $C$ be a free group and $A<B<C$ with $A$ a free factor of $C$. Then $A$ is also a free factor of $B$.  
\end{lemma}

\begin{proof}
    First, $C$ contains $A$ as a free factor, so we can realize $A$ and $C$ as a pair of graphs $\Delta \subset \Gamma$, where $\pi_1(\G,p) \cong C$ for some $p \in \Delta$, and the isomorphism restricts to $\pi_1(\Delta,p) \cong A$.
    Consider the cover $\rho: (\G_B,\tilde{p}) \to (\G,p)$ corresponding to the subgroup $B$ for some $\tilde{p} \in \G_B$. Denoting by $i:\Delta \arr \G$ the inclusion map, we have $i_*(\pi_1(\Delta,p))=A<B=\rho_*(\pi_1(\G_B,\tilde{p}))$, so 
     the inclusion lifts to $\tilde{i}: (\Delta,p) \arr (\G_B,\tilde{p})$. As $\rho \circ \tilde{i}=i$ and $i$ is injective, $\tilde{i}$ is injective. Similarly, $\tilde{i}_*:\pi_1(\Delta,p) \to \pi_1(\G_B,\tilde{p})$ is injective so it follows that $\pi_1(\tilde{i}(\Delta),\tilde{p}) = \tilde{i}_*(\pi_1(\Delta,p)).$ Therefore, $\G_B$ contains $\tilde{i}(\Delta)$, a homeomorphic copy of $\Delta$, and the isomorphism $\rho_*: \pi_1(\G_B,\tilde{p}) \cong B$ restricts to the isomorphism $\rho_*: \pi_1(\tilde{i}(\Delta),\tilde{p})=\tilde{i}_*(\pi_1(\Delta,p)) \cong A$. Therefore, we conclude $A$ is a free factor of $B$.
\end{proof}

Next we show that any automorphism can be approximated on a finite rank free factor by a ``finitely supported" automorphism. Recall $A_n = \<a_i\>_{i=1}^n$ and $B_n = \<a_i\>_{i=n+1}^\infty$.

\begin{lemma} \label{lem:approximation}
For any $\phi\in \Aut(F_{\infty})$ and $n\in \N$, there exists $\psi\in \Aut(F_{\infty})$ and $m\geq n$ such that \begin{enumerate}
    \item $\psi\vert_{A_n}=\phi\vert_{A_n}$,
    \item $\psi\vert_{B_m}=\id$.
\end{enumerate} 
\end{lemma}

\begin{proof}
    Let $m$ be such that $\phi(A_{n}) \subset A_{m}$. Such an $m$ exists and is finite because $\phi(A_n)=\<\phi(a_i)\>_{i=1}^n$ and each $\phi(a_i)$ is a finite word.
    By \Cref{LEM:subfreefact}, $\phi(A_{n})$ is a free factor of $A_{m}$ and $m \ge n$. That is, $A_{m} = \phi(A_{n}) * C$ where $C$ is a free factor of rank $m-n$. Let $b_{1} = \phi(a_{1}),\ldots, b_{n} = \phi(a_{n})$ be a free basis for $\phi(A_{n})$ and let $c_{n+1},\ldots,c_{m}$ be some free basis of $C$. Define $\psi' : A_{m} \rightarrow A_{m}$ to be the change of basis automorphism that sends the standard basis, $\{a_{1},\ldots,a_{m}\}$, of $A_{m}$ to the new basis $\{b_{1},\ldots,b_{n},c_{n+1},\ldots,c_{m}\}$ of $A_{m}$. Extend $\psi'$ by the identity to obtain $\psi:F_\infty \to F_\infty$ as desired.
\end{proof}

\begin{thm}\label{thm:autCB}
The group $\Aut(F_{\infty})$ is coarsely bounded.
\end{thm}

\begin{proof}
    Let $\cU$ be a neighborhood of the identity in $\Aut(F_{\infty})$. As discussed in \Cref{ssec:topology}, we can find some basis element of the form $\cU_{n} \subset \cU$. Define
    \begin{align*}
        f: \begin{cases} a_{i} \leftrightarrow a_{n+i} \quad &\text{if $1 \leq i \leq n$}, \\
        a_{i} \mapsto a_{i}  &\text{if $i>2n$}.
        \end{cases}
    \end{align*}
    We will show that $\Aut(F_{\infty}) = (\cF\cU_{n})^{3}$ where $\cF = \{\id,f\}$.

    Let $\phi \in \Aut(F_{\infty})$. Apply \Cref{lem:approximation} to $\phi^{-1}$ to obtain $\psi \in \Aut(F_\infty)$ and $m\geq n$ such that $\psi\phi \in \cU_{n}$ and $\psi\vert_{B_{m}} = \id\vert_{B_m}$. Let $m' = \max\{m,2n\}$ and define
    \begin{align*}
        g: \begin{cases} a_{i} \leftrightarrow a_{i-n+m'} \quad &\text{if $n+1 \leq i \leq 2n$},\\
        a_{i} \mapsto a_{i} &\text{otherwise}.
        \end{cases}
    \end{align*}
    Our choice of $m'$ ensures that $g \in \cU_{n}$. We can now check that $fg\psi g f \in \cU_{n}$. Indeed, for any $i = 1,\ldots,n$ we have
    \begin{align*}
        fg\psi gf (a_{i}) = fg \psi (a_{m'+i}) = fg (a_{m'+i}) = a_{i}.
    \end{align*}
    Since $f^2=g^2=\id$, it follows that $\psi \in gf\cU_n fg$, so $\psi^{-1} \in gf\cU_n fg \subset \cU_n f \cU_n f \cU_n \subset (\cF \cU_n)^3$. Therefore, $\phi \in \psi^{-1}\cU_n \subset (\cF \cU_n)^3 \cU_n = (\cF \cU_n)^3$, concluding the proof.
\end{proof}

\begin{cor}
The group $\Out(F_{\infty})$ is coarsely bounded.
\end{cor}

\begin{proof}
    Any continuous action of $\Out(F_{\infty})$ on a metric space gives rise to an action of $\Aut(F_{\infty})$ via pre-composing with the quotient map. Since $\Aut(F_{\infty})$ surjects onto $\Out(F_{\infty})$, we have that the orbit of $\Out(F_{\infty})$ must have finite diameter, otherwise it would contradict \Cref{thm:autCB}. 
\end{proof}

\section*{Acknowledgements}
  We would like to thank Mladen Bestvina for comments on an earlier version of this paper. We also thank the referee for several helpful comments.
The authors acknowledge support from NSF grants DMS--1906095 (Hoganson), DMS--1905720 (Domat, Kwak), RTG DMS--1840190 (Domat, Hoganson), and CAREER DMS--2046889 (Kwak).
\bibliography{bib}
\bibliographystyle{plain}

\end{document}